\documentclass[11pt]{article}
\usepackage{amsmath, amssymb, amsthm, mathrsfs, commath, mathtools}
\usepackage[headheight=0pt,headsep=0pt,textwidth=6in,textheight=9in]{geometry}
\usepackage[dvipsnames]{xcolor}

\usepackage{indentfirst}

\usepackage[margin=1cm,font=small]{caption}

\usepackage{tikz}
\usetikzlibrary{decorations.markings}
\usepackage{float}

\usepackage{chngcntr}

\newcommand{\R}{\mathbb{R}}

\newcommand{\C}{\mathbb{C}}
\newcommand{\B}{\mathbb{B}}
\newcommand{\Z}{\mathbb{Z}}
\renewcommand{\S}{\mathbb{S}}

\renewcommand{\phi}{\varphi}

\renewcommand{\d}{\, d}									               
\newcommand{\tn}[1]{\textnormal{#1}}					                   
\renewcommand{\l}{\left}                                               
\renewcommand{\r}{\right}								               
\renewcommand{\b}[1]{\l( #1 \r)}						               
\newcommand{\mc}[1]{\mathcal{#1}}						               
\newcommand{\ip}[1]{\l\langle #1 \r\rangle}				               
\renewcommand{\pd}[2]{\frac{\partial #1}{\partial #2}}                 
\newcommand{\eq}[1]{\begin{equation} #1 \end{equation}}
\newcommand{\aeq}[1]{\begin{align}\begin{split} #1 \end{split}\end{align}}

\newcommand{\II}{\mathrm{I\!I}}

\counterwithin{equation}{section}

\theoremstyle{plain}
\newtheorem{thm}[equation]{Theorem}
\newtheorem{lem}[equation]{Lemma}

\newtheorem{prop}[equation]{Proposition}
\newtheorem{cor}[equation]{Corollary}

\theoremstyle{definition}

\newtheorem{exa}[equation]{Example}
\newtheorem{rmk}[equation]{Remark}
\newtheorem{opr}[equation]{Open Problem}
\newtheorem{dia}[equation]{Diagram}

\title{Stability Estimates for  Commutativity Properties of the Dirichlet-to-Neumann Operator}
\author{Romain Speciel}
\date{\today}

\begin{document}
\maketitle


\abstract{
The Laplacian $\Delta_{\S^{n-1}}$ on the unit sphere $\S^{n-1}\subset \R^n$ has the property that it can explicitly be expressed in terms of $\Lambda$, the Dirichlet-to-Neumann map  of the unit ball, as $\Delta_{\S^{n-1}}=\Lambda^2+(n-2)\Lambda$. In this paper, we seek to characterize those manifolds for which such an exact relationship holds, and more generally measure the discrepancy of such a relationship holding in terms of geometric data. To this end, we obtain a stability estimate which shows that, for a smoothly bounded domain in $\R^3$, if the commutator $[\Lambda,\Delta_{\S^{n-1}}]$ is small then that domain is itself close to a ball. We then study the case of manifolds conformal to the ball, show that a relationship as above implies a radial metric structure, and discuss stability in this setting. Finally, we provide a modern exposition of Gohberg's lemma, a foundational result in microlocal analysis which we employ as a starting step for our reasoning.
}

\vspace{2em}

\section{Introduction}

Let $\Omega\subset \R^n$ be a bounded domain with smooth $(n-1)$-dimensional boundary, assumed nonempty throughout. The \textit{Dirichlet-to-Neumann map} $\Lambda$ maps a function on the boundary to the normal derivative of its harmonic extension inside $\Omega$. Explicitly, it is defined for $f\in C^\infty(\partial\Omega)$ by
\eq{
	\Lambda f=\partial_\nu u,
}
where $u\in C^\infty(\Omega)$ solves
\eq{
	\begin{cases}
		\Delta u\coloneqq -(\partial_{x_1}^2+\dots+\partial_{x_n}^2)u=0 \quad\tn{in }\Omega,\\
		u|_{\partial\Omega}=f,
	\end{cases}
}
and $\nu$ is the outward pointing normal to $\partial\Omega$. This map extends to a bounded linear operator $\Lambda\colon H^{1/2}(\partial\Omega)\to H^{-1/2}(\partial\Omega)$, where $H^s(M)$ denotes the standard $L^2$ Sobolev spaces on a Riemannian manifold $(M,g)$. It is a central object of study in the context of geometric inverse problems, finds applications in fields ranging from medical imaging to geology, and has been extensively investigated over the last half-century. In particular, this operator is the object of interest in the \textit{Calder\'on problem}, which asks if the interior of a manifold may be determined from the knowledge of its Dirichlet-to-Neumann map (consult \cite{U14} for a broad overview of the field).

When $\Omega=\B^n$ is the unit ball with boundary $\S^{n-1}$, the spherical symmetry allows an explicit computation via decomposition into eigenfunctions. One finds that
\eq{
	\Lambda^2+(n-2)\Lambda =\Delta_{\S^{n-1}},
}
where $\Delta_{\S^{n-1}}$ is the Laplacian on the unit sphere with respect to the standard induced metric (see \cite{CGGS23} for extensive, detailed computations, as well as a overview of the spectral properties of the Dirichlet-to-Neumann map). Furthermore, in the model case of $\R^{n-1}$ viewed as the boundary of $\R^n_+$, the Laplacian on $\R^{n-1}$ is exactly the square of the Dirichlet-to-Neumann map (see \cite{CS07} for details of this construction and an extension to the setting of the fractional Laplacian). Both of these examples display a perhaps surprising relationship between the boundary Laplacian, a local differential operator, and Dirichlet-to-Neumann map, a nonlocal pseudodifferential operator highly sensitive to the interior geometry

These observations motivate the two following questions, which we pose for a general Riemannian manifold $(M,g)$ with boundary: 

\vspace{0.3em}

Question I. When is $\Lambda$ exactly a function of $\Delta_{\partial M}$?

\vspace{0.3em}
Question II. More generally, when do $\Lambda$ and $\Delta_{\partial M}$ commute?

\vspace{0.3em}

\noindent This latter question was originally studied in \cite{GKLP22}, where the authors consider the question of whether the Dirichlet-to-Neumann map commutes with the boundary Laplacian for any Euclidian domains other than the ball (note commutativity is equivalent to simultaneous diagonalizability since both operators are self adjoint). Their results, in conjunction with the work in \cite{S25}, demonstrate the following:

\begin{thm}[\cite{GKLP22, S25}]
\label{thm: commutator means ball}
	Let $\Omega\subset \R^n$, $n\geq 3$, be a bounded domain with nonempty smooth boundary. Then,
	\eq{
		[\Lambda,\Delta_{\partial\Omega}]=0 \quad \iff \quad \tn{$\Omega$ is a ball}.
	}
\end{thm}

\noindent Interestingly, the two dimensional case, studied in \cite{S25}, is exceptional (as is often the case in the study of the Dirichlet-to-Neumann operator) as there exist non-round simply connected domains in the plane with this commutativity property.

In this paper, we study the stability properties associated to Question I and Question II above both for Euclidian domains and for manifolds conformal to the standard unit ball. Section \ref{sec: euclidian} considers the Euclidean case and establishes in this setting a stability estimate corresponding to Theorem \ref{thm: commutator means ball} , as conjectured in \cite{GKLP22}, using microlocal techniques and geometric inequalities. To state our results, we first fix explicit norms for our function spaces. Define for any Riemannian manifold $(M,g)$, possibly with boundary, the spaces $L^2(M)$ and $H^s(M)$, $s>0$, as the closure of $C^\infty(M)$ under the norms
\eq{
	\norm{f}_2=\norm{f}_{L^2(M)}=\sqrt{\int_M\abs{f}^2 \d V_g} \quad \tn{and} \quad \norm{f}_{H^s(M)}=\norm{(1+\Delta_M)^{s/2}f}_2,
}
\noindent respectively, where $dV_g$ is the volume form induced by $g$. We establish the following result.

\begin{thm}
\label{thm: euclidian stability}
	Let $\Omega\subset \R^3$ be a bounded domain diffeomorphic to $\B^3$ with smooth boundary of area $4\pi$. For sufficiently small $\epsilon>0$, if
	\eq{
		\norm{[\Lambda,\Delta_{\partial\Omega}]}_{H^1(\partial\Omega)\to L^2(\partial\Omega)}<\epsilon,
	}
	then there exists a conformal parametrization $\psi\colon \S^2\to \partial\Omega$ and a vector $c_{\partial\Omega}\in \R^3$ satisfying
	\eq{
		\norm{\psi-(c_{\partial\Omega}+\iota)}_{H^2(\S^2)}< C \epsilon,
	}
	where $\iota\colon \S^2\to \R^3$ is the standard embedding and $C$ is a universal constant.  
\end{thm}

\noindent The careful reader notes that the result comes with additional assumptions when compared to Theorem \ref{thm: commutator means ball}. Indeed, we restrict our attention exclusively to domains in $\R^3$ which are diffeomorphic to the ball. We conjecture in Open Problem \ref{opr: other dims} that both these dimensional and topological assumptions, which we use to apply Topping's inequality in Lemma \ref{lem: diameter bound} and to guarantee the existence of an umbilical point on the boundary throughout Section \ref{sec: euclidian}, can in fact be removed. One also notices that we view $[\Lambda, \Delta_{\partial\Omega}]$ as mapping $H^1({\partial\Omega})\to L^2({\partial\Omega})$, despite its initial appearance as an operator of order two. In fact, as explained in detail in the proof of Proposition \ref{prop: gohberg to geometric}, the principal symbol of this commutator always vanishes, thus justifying this choice.

Section \ref{sec: potential} then treats manifolds conformal to the unit ball, obtaining both uniqueness and stability results by combining harmonic analysis with complex geometric optic (CGO) methods standard in the study of the Calder\'on problem. It is well known that studying the Dirichlet-to-Neumann map across conformal metrics is equivalent to studying the corresponding operator when extending to the interior with respect to $\Delta+q$, for some smooth potential $q$, instead. We therefore briefly describe here the standard reductions which convert this problem to studying the Schr\"odinger equation in the ball, then state our results in this latter setting.

To this end, let $M=\B^n$, $n\geq 3$, be equipped with the metric $g_\phi=e^{2\phi}g_\tn{std}$. This manifold is conformal to the standard metric on the ball via the function $\phi\colon M\to \R$, which we suppose to vanish in a neighborhood of the boundary. A classical computation shows that
\eq{
	\Delta_\phi u=e^{-2\phi}\b{\Delta u +(n-2)\nabla\phi\cdot \nabla u}=0\iff \nabla \b{e^{(n-2)\phi}\cdot \nabla u}=0,
}
hence we obtain the same Dirichlet-to-Neumann map when considering extensions with respect to the operator $\nabla(\gamma\cdot \nabla u)$ with conductivity function $\gamma=e^{(n-2)\phi}$ (which, importantly, never vanishes). The Dirichlet-to-Neumann map of this divergence-form operator itself corresponds to that with respect to extensions by solutions to the Schr\"odinger equation
\eq{
	(\Delta+q)u=0 \quad\tn{with}\quad q=\frac{\Delta\sqrt{\gamma}}{\sqrt{\gamma}},
}
which yields a more tractable problem. We denote this operator by $\Lambda_q$. Explicitly, $\Lambda_qf=g$ if an only if there exists a function $u\in C^\infty(\B^n)$ such that

\eq{
	(\Delta+q)u=0\quad\tn{and}\quad \begin{cases}
		u|_{\S^{n-1}}=f\\
		\pd{u}{r}\big{|}_{\S^{n-1}}=g
	\end{cases}.
}

The $L^2$ stability estimates for $\Lambda_q$ imply $L^\infty$ stability estimates for the corresponding coefficient $\gamma$, and hence for the conformal factor $\phi$ as well. This process is outlined in detail in \cite{A88}. We therefore center our attention on studying how the analytic properties of the operator $\Lambda_q$, in the context of Question I and Question II, imply geometric symmetry of the potential $q$.

To precisely state the results, define first the projection
\eq{
\label{eq: radial projection}
	P\colon L^2(\B^n)\to L^2(\B^n)\quad\tn{by}\quad  Pf=\int_{SO(n)}f\circ R\d V,
}
where $V$ is the normalized Haar measure on $SO(n)$. We address Question I in Proposition \ref{prop: q1 for potentials} and prove that the potential $q$ is radial if and only if $\Lambda$ is a function of $\Delta_{\S^{n-1}}$ (hence the Dirichlet-to-Neumann map of a manifold conformal to the ball is a function of the boundary Laplacian if and only if the conformal factor is radial). Furthermore, we obtain the following corresponding stability estimate:

\begin{thm}
\label{thm: conf q1 stability}
	Suppose that $\norm{q}_{H^s(\B^n)}<M$ with $s>n/2$ and $M>0$, and say $\Lambda_q$ may be written as $\Lambda_q = f(\Delta_{\S^{n-1}})+E$ for some function $f\colon \R\to \R$ and ``error term'' $E$, a bounded map $E\colon H^{1/2}(\S^{n-1})\to H^{-1/2}(\S^{n-1})$ whose norm as an operator between these spaces we denote $\norm{E}_*$. If $\norm{E}_*$ is sufficiently small, then
	\eq{
		\norm{q-Pq}_2\leq C\cdot \omega(\,\norm{E}_*),
	}
	where $C$ depends on $M$, $s$ and $n$, and $\omega$ is a logarithmic modulus given by
	\eq{
		\omega(t)= \abs{\log t}^{-\delta} \quad\tn{for some }0<\delta<1\tn{ depending only on $n$}.
	}
\end{thm}


\noindent Note such logarithmic stability is standard in this type of inverse problem. We then turn our attention to Question II, and prove the following result, akin to Theorem \ref{thm: euclidian stability} in the context of infinitesimal conformal perturbations:

\begin{thm}
\label{thm: conformal inf q2}
	For $t$ in some interval $(-\epsilon, \epsilon)$ with $\epsilon >0$, consider a smoothly varying family of potentials $q_t$ on $\B^{n\geq3}$ with $q_0\equiv 0$. Writing $q'=\pd{}{t}\big{|}_{t=0}q_t$ and $\Lambda'\coloneqq \pd{}{t}\big{|}_{t=0}\Lambda_{q_t}$, we have
	\eq{
		[\Lambda',\Delta_{\S^{n-1}}]=0\quad\tn\quad\iff \quad q'=Pq'.
	}
\end{thm}

\noindent In fact, we conjecture in Open Problem \ref{opr: commute iff radial} that the above theorem holds beyond the perturbative setting, so every conformal factor which induces a Dirichlet-to-Neumann map commuting with the Laplacian on the boundary sphere is itself radial.

Finally, Appendix \ref{sec: Gohberg lemma} provides a detailed, modern  exposition of Gohberg's lemma, a classical result in microlocal analysis which is nonetheless seemingly hard to find in the present literature. This result is a crucial step in Section \ref{sec: euclidian}.

\vspace{1em}

\noindent \textbf{Acknowledgements:} The author thanks Otis Chodosh, Josef Greilhuber, Rafe Mazzeo, Gregory Parker and Iosif Polterovich for many helpful conversations throughout the project. This work was in part supported by an NSERC-PGSD grant.

\vspace{5em}

\section{Stability among Euclidian domains}
\label{sec: euclidian}

In this section, we focus on understanding the stability aspect of Question II for a bounded Euclidian domain $\Omega\subset \R^3$ with smooth boundary. The main result here is Theorem \ref{thm: euclidian stability}, whose proof proceeds in two mains steps. First, we apply Gohberg's lemma (see Appendix \ref{sec: Gohberg lemma}) to convert the analytic assumption on the norm of the commutator into a geometric condition in terms of the gradient of the second fundamental form $\II$ of the boundary, in Proposition \ref{prop: gohberg to geometric}. The second step is to convert this bound into the desired geometric stability. We deduce in Lemma \ref{lem: diameter bound} an intrinsic diameter estimate by applying Topping's inequality and the Bonnet-Meyers theorem. Then, by combining the diameter bound with the fact that the gradient of the second fundamental form is controlled, we obtain that $\partial \Omega$ is a ``nearly umbilical surface'' and use a well known result from \cite{DLM05} to conclude that it must in fact be close to a sphere, thus establishing Theorem \ref{thm: euclidian stability}.

To begin, we obtain from the assumptions of Theorem \ref{thm: euclidian stability} a bound on the gradient of the second fundamental form.

\begin{prop}
\label{prop: gohberg to geometric}
	With $\Omega\subset \R^n$ a bounded domain with nonempty smooth boundary,
	\eq{
		\norm{[\Lambda,\Delta_{\partial \Omega}]}_{H^1(\partial \Omega)\to L^2(\partial \Omega)}<\epsilon\quad \implies\quad \norm{\nabla \II}_\infty<C_n\epsilon,
	}
	where $C_n$ is a constant depending only on the ambient dimension $n$.
\end{prop}

\begin{proof} We first recall several fundamental facts, following \cite{GKLP22}, about the involved operators. The Dirichlet-to-Neumann map is a pseudodifferential operator of order one which may be written as
	\eq{
	\label{eq: DtoN expression as sqrt}
		\Lambda=\sqrt{\Delta_{\partial \Omega}}+B,
	}
	where $B$ is an order zero pseudodifferential operator. The principal symbol of $B$ is expressed in terms of the second fundamental form $\II$ and the mean curvature $H=\tn{tr}( \II)/(n-1)$ of ${\partial \Omega}$ as
	\eq{
		\sigma_0(B)(x,\xi)=\frac{1}{2}\b{\frac{\II(\xi,\xi)}{\abs{\xi}^2}-(n-1)H},
	}
	as explained in Section 2.1 of \cite{GKLP22}. From Equation (\ref{eq: DtoN expression as sqrt}), we note that the commutator $[\Lambda, \Delta_{\partial \Omega}]$ is in fact only of order $1$ (as noted in th introduction), and its principal symbol is thus given by the Poisson bracket $\{\sigma_0(B)(x,\xi),\abs{\xi}^2\}$. Choosing coordinates $x$ about a given point $\tilde x\in \partial \Omega$ so that the induced metric coefficients satisfy $g_{ij}(\tilde x)=\delta_{ij}$ and the first order derivates of the metric tensor vanish at $\tilde x$, we expand the Poisson bracket to obtain
	\eq{
		\sigma_1[\Lambda, \Delta_{\partial \Omega}]\b{x,\xi=(\xi_1,\dots, \xi_{n-1})}=\sum_{i}\xi_i\b{\sum_{j,k}\frac{\xi_j\xi_k}{\abs{\xi}^2}\pd{\II_{jk}}{x_i}-(n-1)\pd{H}{x_i}},
	}
	as in \cite{GKLP22}.
	Now, since $\norm{[\Lambda, \Delta_{\partial \Omega}]}_{H^1({\partial \Omega})\to L^2({\partial \Omega})}<\epsilon$ by assumption, a direct application of Gohberg's lemma (or, more precisely, of Corollary \ref{cor: Gohberg for any order}) yields
	\eq{
		\abs{\abs{\xi}^{-1}\sum_{i}\xi_i\b{\sum_{j,k}\frac{\xi_j\xi_k}{\abs{\xi}^2}\pd{\II_{jk}}{x_i}-(n-1)\pd{H}{x_i}}}<\epsilon,
	}
	which we rearrange to obtain
	\eq{
	\label{eq: poly with curvature}
		\abs{\sum_{i,j,k}\xi_i\xi_j\xi_k\pd{\II_{jk}}{x_i}-(n-1)\sum_{i,j}\xi_i\xi_j\xi_j\pd{H}{x_i}}<\epsilon \abs{\xi}^3.
	}
	At this point, we shall combine three observations:
	\begin{enumerate}
		\item The ambient space $\R^n$ is flat so the Codazzi equation gives that the coefficients $\partial \II_{jk}/\partial x_i$ are symmetric in $i,j,k$.
		\item Fixing $x$, the left had side of Equation (\ref{eq: poly with curvature}) is a polynomial in $\xi$ which is bounded in absolute value by $\epsilon$ on the unit disc. We may define two norms on the vector space of homogeneous polynomials in $n$ variables of degree $3$: the maximum absolute value of the function on the unit disc, or the maximum absolute value of the coefficients of each monomial. Since norms on finite dimensional vector spaces are all equivalent, we deduce that the monomial coefficients are bounded by $\epsilon$, up to a dimensional constant.
		\item Since $H=\sum \II_{jj}/(n-1)$ by definition, we have that 
	\eq{
		\pd{H}{x_i}=\frac{1}{n-1}\sum_j \pd{ \II_{jj}}{x_i},
	}
	and can substitute this expression in Equation \ref{eq: poly with curvature}.

	\end{enumerate}
	Altogether, we conclude
	\eq{
		\abs{\pd{\II_{ij}}{x_k}}<C_n\epsilon
	}
	for some dimensional constant $C_n$, as claimed.
\end{proof}

\begin{rmk}
	In \cite{GKLP22}, the authors aim to deduce that the mean curvature of ${\partial \Omega}$ is constant from the assumption that $[\Lambda, \Delta_{\partial \Omega}]=0$. To this end, they follow a similar reasoning to that above but instead obtain a linear system relating the coefficients $\partial \II_{jj}/\partial x_i$ to each other, and then conclude $\nabla H=0$. Equation (\ref{eq: poly with curvature}) is a generalization of this reasoning, and allows us to obtain the stronger result about the entire second fundamental form.
\end{rmk}

Inspired by the proof of Theorem 1.3 in \cite{GKLP22}, we may seek to use Proposition \ref{prop: gohberg to geometric} to bound the gradient of the mean curvature of $\Sigma$ in hope to then apply a stability-type result for the Alexandrov soap bubble theorem, such as the one proved in \cite{MP19}, to conclude the desired result. However, as the following example shows, it is not sufficient to simply consider $\nabla H$; one must look at the entire second fundamental form.

\begin{exa}
\label{exa: delaunay surfaces}
	Introduced nearly 200 years ago by Charles-Eug\`ene Delaunay in \cite{D41}, \textit{Delaunay surfaces} are non-compact surfaces of revolution in $\R^3$ with constant mean curvature. These surfaces are a classical object of study which display an early example of a bubbling phenomenon as the mean curvature is brought to zero. By carefully capping off a Delaunay surface to make it diffeomorphic to $\S^2$, we can obtain a surface with $\abs{\nabla H}<\epsilon$ for any positive $\epsilon$ which is nonetheless very far from the standard sphere. We briefly outline this construction, which is a simple smoothing-and-gluing argument.
	
	Consider an ellipse, with major axis $1$ and minor axis $\epsilon>0$, placed atop the $x$-axis in $\R^2$ so that its major axis aligns with the $y$-axis. Then, roll this ellipse along the $x$-axis and denote the path of its (initially) upper focus by $u_\epsilon(x)$, as in Figure 1 of Diagram \ref{dia: delaunay}. Let $u_0(x)=\sqrt{1-x^2}$ denote the limiting path, a semicircle. The Delaunay surfaces are exactly those surfaces of revolution which result from rotating $u_\epsilon$ about the $x$-axis in $\R^3$. These surfaces notably have constant mean curvature (and are in fact the only surfaces of revolution with this property).
	
	Now, take $\phi$ a smooth step function so that $0\leq \phi\leq 1$ with $\phi(x)=0$ when $x<-1/2$ and $\phi(x)=1$ when $x>1/2$, and define $S_\epsilon$ to be the surface of revolution obtained by rotating the graph over $[-1,\infty)$ of 
	\eq{
		v_\epsilon(x)\coloneqq(1-\phi(x))u_0(x)+\phi(x)u_\epsilon(x)
	}
	around the $x$-axis. The mean curvature of $S_\epsilon$ is constant away from the region $x\in[-1/2,1/2]$, since it agrees with a Delaunay surface. Furthermore, since $v_\epsilon\to u_0$ on this interval in the $C_3$ norm as $\epsilon \to 0$, and the gradient of the mean curvature is controlled by the $C_3$ norm, we deduce that by taking $\epsilon$ sufficiently close to zero we may ensure that the gradient of the mean curvature of $S_\epsilon$ is arbitrarily small. By repeating this capping procedure on the other end of $S_\epsilon$, we obtain a surface diffeomorphic to $\S^2$ with $|\nabla H|$ arbitrarily small which is nonetheless not close to a sphere. An sketch of such a surface is displayed in Figure 2 of Diagram \ref{dia: delaunay}. Note in particular that this construction may be performed for any prescribed area by capping at a fixed length and then rescaling.
\end{exa}

\begin{figure}[h!]
\label{fig: log ovals}
    \centering
    \begin{minipage}{0.45\textwidth}
        \centering
        \includegraphics[width=0.9\textwidth]{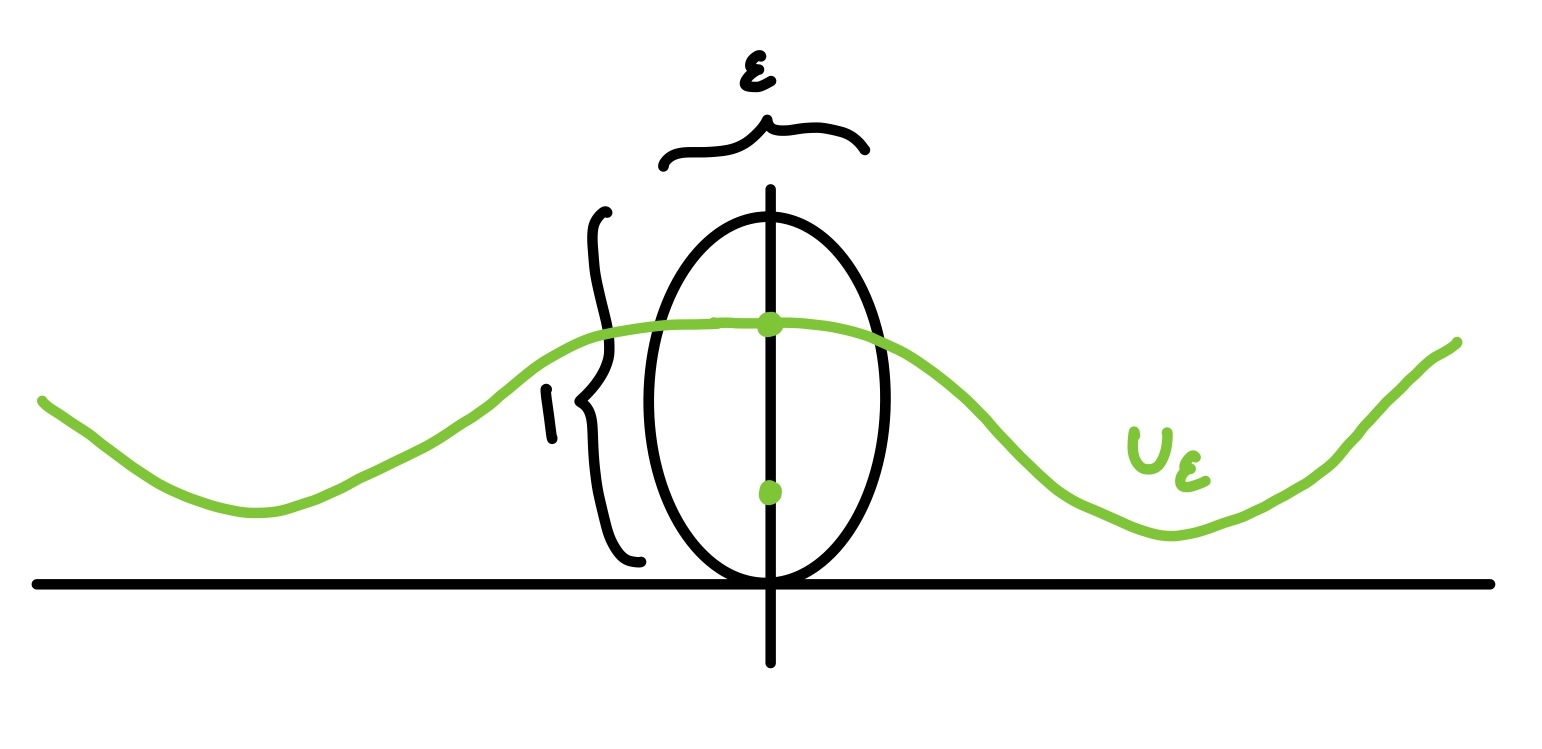} 
        \caption{}
    \end{minipage}
    \begin{minipage}{0.45\textwidth}
        \centering
        \includegraphics[width=0.55\textwidth]{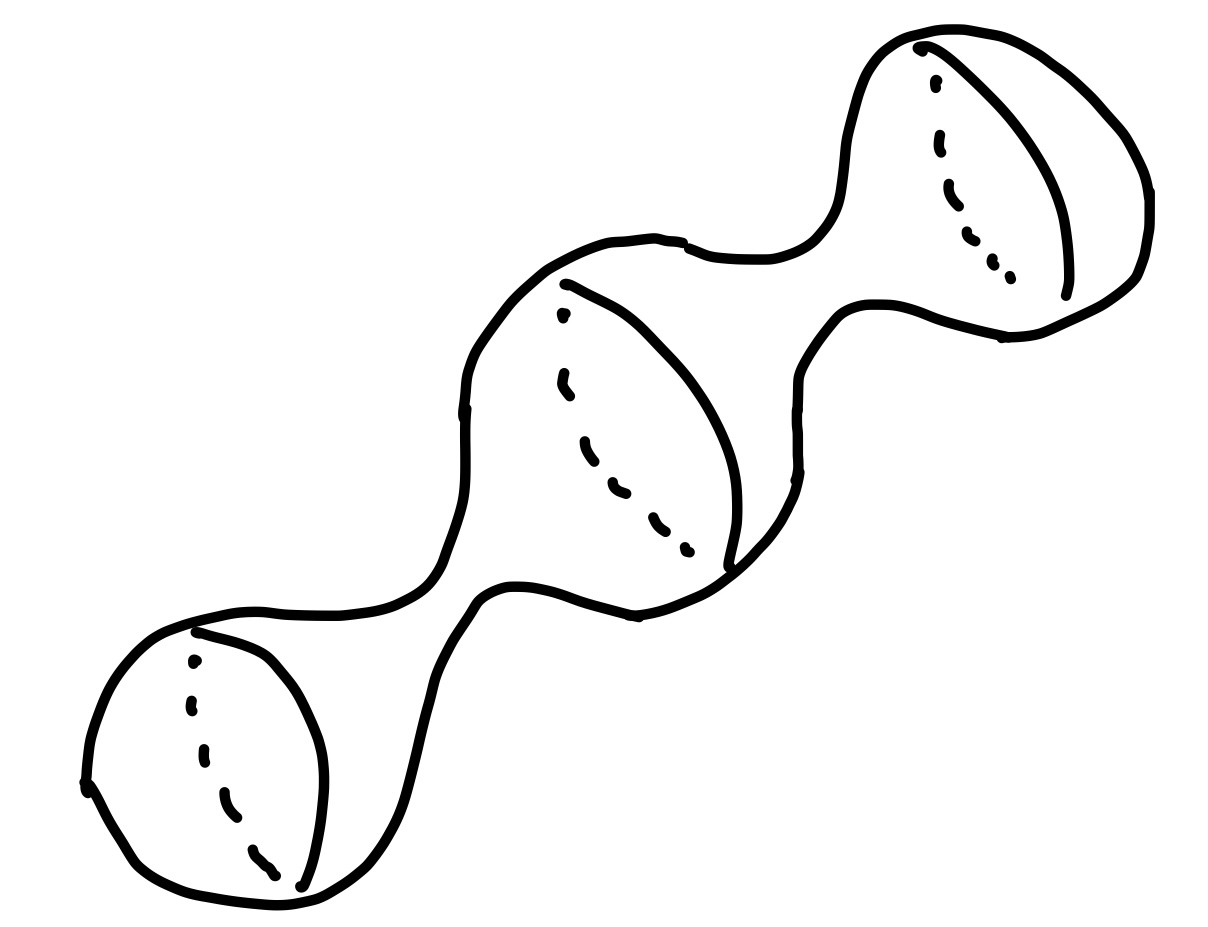} 
        \caption{}
    \end{minipage}
\end{figure}
\vspace{-1em}

\begin{dia}
\label{dia: delaunay}
	Two illustrations of the constructions described in Example \ref{exa: delaunay surfaces}. Figure 1 displays the cross section of a Delaunay surface, and Figure 2 displays a capped Delaunay surface, which is therefore diffeomorphic to $\S^2$.
\end{dia}

\vspace{1em}

One must therefore use the full information given by the commutator bound on the gradient of the second fundamental form, not just consider mean curvature. In what follows, we further restrict our attention to surfaces in $\R^3$ homeomorphic to the sphere, as discussed in the introduction. These assumptions ensure the existence of an umbilical point and allow us to apply the necessary inequalities (Remark \ref{rmk: dimensional discussion} below addresses this in more detail). The following key lemma yields a diameter bound from these assumptions.

\begin{lem}
\label{lem: diameter bound}
	Let $\Sigma\subset \R^3$ be a smooth embedding of $\S^2$ of area $4\pi$ and suppose that $\abs{\nabla \II}<1/150$ at every point of $\Sigma$. Then the diameter $d$ of $\Sigma$ is less than $2025$.
\end{lem}
\begin{proof} We begin by recalling the classical fact that, as a consequence of the Poincar\'e-Hopf theorem, every smooth embedding of $\S^2$ in $\R^3$ has at least one umbilical point, meaning a point where the principal curvatures are equal to each other. Let therefore $p\in \Sigma$ be umbilical with principal curvatures $k_1=k_2=k\in \R$. To control the diameter of $\Sigma$, we combine two bounds, the first addressing the case of small $\abs{k}$ and the second the case of large $\abs{k}$. To this end, suppose initially that $\abs{k}\leq 2$. The mean curvature $H$ is half the trace of the second fundamental form, hence we may apply the assumption that $|\nabla\II|<150$ to bound its absolute value above by $2+d\cdot 1/150$. Topping's inequality \cite{T08} then implies that
	\eq{
	\label{eq: topping}
		d\leq \frac{32}{\pi}\int_\Sigma\abs{H}\leq 128\b{2+d/150}\implies d\leq 2025,
	}
	as desired.
	
	On the other hand, suppose instead that $\abs{k}>2$. Recall that eigenvalues of symmetric $2\times 2$ matrices depend in a Lipschitz fashion on its entries, with Lipschitz constant $2$. Therefore, we can bound the oscillation of the principal curvatures, as they are the eigenvalues of the second fundamental form. Deduce that in the intrinsic ball $B(p,\pi)\subset \Sigma$ of radius $\pi$ centered at $p$, the principal curvatures are bounded below by $2-\pi/75$ in absolute value, as a consequence of the assumed gradient bound on $\II$. The Gauss curvature $K$ is thus bounded below by $3$ in $B(p,\pi)$. Following the proof of the Bonnet-Meyers theorem \cite{C92}, we deduce that geodesics of length greater than $\pi/\sqrt{3}$ cannot be length minimizing in $B(p,\pi)$, and therefore $B(p,\pi)=\Sigma$ so the diameter bound is established.
\end{proof}

\begin{rmk}
\label{rmk: dimensional discussion}
	Note that the dimensional restriction is used twice in the above proof. First, we require an umbilical point to propagate our estimates, and the existence of such a point is not guaranteed in higher dimensions. Second, the integrand on the right hand side of Topping's inequality has an exponent which is 1 in the surface case, but grows with the dimension of the submanifold to which it is applied. The conclusion in Equation \ref{eq: topping} does not hold when the right hand side of the first string of inequalities is not linear.
\end{rmk}

We thus deduce that if the norm of the commutator $[\Lambda, \Delta_{\partial \Omega}]$ is sufficiently small, the diameter of $\partial \Omega$, and thus the oscillation of the second fundamental form, is controlled. The final step is then to conclude that this is enough to guarantee proximity to the sphere. To this end, we shall use the following well known result of De Lellis and M\"uller.

\begin{thm}[cf. Theorem 1.1 of \cite{DLM05}]
\label{thm: DLM}
	Let $\Sigma\subset \R^3$ denote a smooth closed surface of area $4\pi$. If
	\eq{
		\norm{\II- H\cdot \tn{Id}}_2^2<8\pi,
	}
	then there exists a conformal parametrization $\psi\colon \S^2\to \Sigma$ and a vector $c_\Sigma\in \R^3$ such that
	\eq{
		\norm{\psi-(c_\Sigma +\iota)}_{H^2(\S^2)}\leq C\norm{\II- H\cdot \tn{Id}}_2,
	}
	where $\iota\colon \S^2\to \R^3$ is the standard embedding and $C$ is a universal constant.  
\end{thm}


\begin{rmk}
	By applying the above result, we chose to employ the stability of \textit{nearly umbilical surfaces}, meaning surfaces whose second fundamental form is close to the mean curvature times the identity, rather than the stability of Alexandrov's soap bubble theorem, as described in \cite{MP19}, for example. This is because the former guarantees a stronger sense of closeness in the $H^2(\S^2)$ sense, rather than the $L^\infty$ closeness of the latter. However, in practice, both approaches are reasonable ways to conclude (since the diameter estimate obtained above now enables us to closely control the norm of the mean curvature globally), and one could certainly use the latter if interested in $L^\infty$ closeness. Note also that Theorem \ref{thm: DLM} is stated only for surfaces in $\R^3$, so we invoke also here the dimensional assumption. 
\end{rmk}

It remains to show that the requisites of Theorem \ref{thm: DLM} are indeed satisfied in our setting. We prove this in the case of an arbitrary surface $\Sigma$, and specify back to the boundary of our domain only in the final proof of the theorem.

\begin{prop}
\label{prop: nearly umbilical}
	Let $\Sigma\subset \R^3$ be a smooth embedding of $\S^2$ of area $4\pi$ and suppose that $\abs{\nabla \II}<\epsilon$ at every point of $\Sigma$, for $\epsilon>0$ sufficiently small. Then,
	\eq{
		\norm{\II-H\cdot \tn{Id}}_\infty<\tilde{C}\epsilon,
	}
	where $\tilde{C}$ is a universal constant.
\end{prop}
\begin{proof}
	Again, let $p$ be an umbilical point of $\Sigma$ with principal curvatures $k_1=k_2=k\in 
	\R$, and suppose $\epsilon$ is sufficiently small so as to ensure that $3\epsilon d<1/2$ (which is possible precisely by Lemma \ref{lem: diameter bound}). Reasoning as in Lemma \ref{lem: diameter bound}, we observe that if $\abs{k}>1+3\epsilon d$, then the principal curvatures are bounded below in absolute value by $1+\epsilon d$, so in particular the Gauss curvature $K$ is everywhere strictly greater than $1$. However, by the Gauss-Bonnet theorem,
	\eq{
		\int_\Sigma K=4\pi,
	}
	from which we obtain a contradiction. If $|k|<1-3\epsilon d$, we deduce a similar contradiction. Therefore, at the umbilical point $p$, we have that $\abs{\abs{k}-1}\leq 3\epsilon d$, and in fact $k$ must also necessarily be positive since otherwise the mean curvature would be negative everywhere on $\Sigma$, which is impossible for embedded surfaces. Therefore, at $p$, the second fundamental form $\II_p$ satisfies $\abs{\II_p-\tn{Id}}<3\epsilon d$. Using this along with the global diameter bound and the assumption that $\abs{\nabla \II}<\epsilon$, it follows that
	\eq{
		\norm{\II-H\cdot \tn{Id}}_\infty\leq \norm{\II-\tn{Id}}_\infty +\norm{H-1}_\infty\leq \tilde C \epsilon
	}
	for some constant $\tilde C$ independent of $\Sigma$, as desired.
\end{proof}

We are now finally ready to prove the main theorem:

\begin{proof}[Proof of Theorem \ref{thm: euclidian stability}]
	By Proposition \ref{prop: gohberg to geometric}, we have that $\norm{\nabla \II}_\infty<C_3\epsilon$ (here $C_3$ is the dimensional constant from the conclusion of Proposition \ref{prop: gohberg to geometric} in ambient dimension $n=3$) and, taking $\epsilon$ sufficiently small, we may then apply Proposition \ref{prop: nearly umbilical} to deduce that
	\eq{
		\norm{\II- H\cdot \tn{Id}}_2\leq \tilde C\epsilon
	}
	since the area of $\partial \Omega$ is assumed to be $4\pi$, so the $L^\infty$ bound may be converted to an $L^2$ bound up to a constant. We now apply Theorem \ref{thm: DLM} and the result immediately follows.
\end{proof}

We still expect some form of Theorem \ref{thm: euclidian stability} to hold in higher dimensions and for broader topologies, and pose this as an open problem below.

\begin{opr}
\label{opr: other dims}
	Let $\Omega\subset \R^{n\geq 4}$ be a bounded domain with smooth boundary of $(n-1)$-volume $1$. Show that for sufficiently small $\epsilon >0$, if
	\eq{
		\norm{[\Lambda,\Delta_{\partial\Omega}]}_{H^1(\partial\Omega)\to L^2(\partial\Omega)}<\epsilon,
	}
	the $\Omega$ is close to the ball in the appropriate sense.

\end{opr}

\vspace{5em}

\section{Stability among manifolds conformal to the ball}
\label{sec: potential}

We now turn our attention to the class of manifolds which are conformal to the unit ball, and investigate Question I and Question II in this setting. As discussed in the introduction, this problem leads to the study of the Dirichlet-to-Neumann map $\Lambda_q$ with respect to the interior extension operator $\Delta+q$. The radiality of the potential $q$ is precisely the property crucial to answering our questions. As a first step to establishing this, we answer Question I below.

\begin{prop}
\label{prop: q1 for potentials}
	The operator $\Lambda_q$ is a function of the boundary Laplacian $\Delta_{S^{n-1}}$ if and only if $q$ is radial.
\end{prop}
\begin{proof}
	If $q$ is radial, then $\Lambda_q$ commutes with all actions of $SO(n)$ on $\S^{n-1}$. But each eigenspace of the Laplacian on the $(n-1)$ sphere is an irreducible representation of $SO(n)$, and so a standard application of Schur's Lemma implies that $\Lambda_q$ is a constant multiple of the identity on each eigenspace, and thus a function of $\Delta_{\S^{n-1}}$.
	
	Conversely, suppose $\Lambda_q=f(\Delta_{\S^{n-1}})$ for some function $f\colon \R\to \R$. Then, since $\Delta_{\S^{n-1}}$ commutes with all actions of $SO(n)$, so does $\Lambda_q$. But, crucially, a simple change of variables shows
	\eq{
	\label{eq: commute dton with rotation}
		\Lambda_{q}=(R^{-1})^*\Lambda_{q\circ R} R^*\quad
	}
	for any $R\in SO(n)$, so our assumption in fact implies that $\Lambda_q=\Lambda_{q\circ R}$ for every $R\in SO(n)$. From this follows that $q=q\circ R$ for every $R\in SO(n)$ (by the well-known solution of the Calder\'on problem in the conformal case \cite{SU87}), hence $q$ is radial.
\end{proof}

We now seek to establish Theorem \ref{thm: conf q1 stability}, a stability result corresponding to Proposition \ref{prop: q1 for potentials}. To this end, we employ the following well known result, which gives a stability estimate to the Calder\'on problem in the case of manifolds conformal to Euclidian domains, by applying the CGO techniques developed in \cite{SU87}:

\begin{thm}[cf. Proposition 1 in \cite{A88}] 
\label{thm: calderon stability}
	Suppose that $\norm{q_i}_{H^s(\B^n)}<M$, $i=1,2$, with $s>n/2$ and $M>0$. Then,
	\eq{
		\norm{q_1-q_2}_{L^2(\B^n)}\leq C \cdot \omega \b{\,\norm{\Lambda_{q_1}-\Lambda_{q_2}}_*},
	}
	where $C$ and $\omega$ are as in Theorem \ref{thm: conf q1 stability}.
\end{thm}

Recall from the introduction that $P$ is the projection onto radial functions, as defined in Equation \ref{eq: radial projection}, and $\norm{-}_*$ denotes the norm of an operator as a map from $H^{1/2}(\S^{n-1})\to H^{-1/2}(\S^{n-1})$.

\begin{proof}[Proof of Theorem \ref{thm: conf q1 stability}]
	Following the observation in Equation (\ref{eq: commute dton with rotation}), we note that for any $R\in SO(n)$,
	\eq{
		\Lambda_{q\circ R}=R^*\Lambda_{q} (R^{-1})^*=f(\Delta_{\S^n})+R^*E(R^{-1})^*.
	}
	We directly apply Theorem \ref{thm: calderon stability} and deduce
	\eq{
		\norm{q-q\circ R}_{L^{2}(\B^n)}\leq C\cdot\omega\b{\norm{\Lambda_q-\Lambda_{q\circ R}}_*}=C\cdot\omega\b{\norm{E-R^*E(R^{-1})^*}_*}.
	}
	Now, to address the right hand side, we note that $\omega$ is subadditive when the input is sufficiently small, and therefore
	\eq{
		\omega\b{\norm{E-R^*E(R^{-1})^*}_*}\leq 2\,\omega(\,\norm{E}_*).
	}
	We with to relate the left hand side to $Pq$. Integrate the identity
	\eq{
		\norm{q-R^*q}_2^2=\norm{q}_2^2+\norm{R^*q}_2^2-2\ip{q,R^*q}
	}
	over all $R\in SO(n)$ to obtain
	\eq{
		\int_{SO(n)} \norm{q-R^*q}_2^2\d H = 2\norm{q}_2^2-2\ip{q,Pq}=2\norm{q-Pq}_2^2,
	}
	where $d H$ denotes the normalized Haar measure on $SO(n)$ and the last step follows from the fact that $P$ is an orthogonal projection. We therefore conclude that 
	\eq{
		\norm{q-Pq}_2\leq C\cdot  \omega(\,\norm{E}_*),
	}
	as desired.
\end{proof}

Next, we turn our attention to Question II. In the conformal setting, this question is significantly more subtle. Our main result here is a characterization of the those potentials which infinitesimally preserve the commutativity property. To this end, we begin with the following lemma, which is a refinement of the well known fact that the span of the product of harmonic functions is dense in the ball.

\begin{lem}
\label{lem: radial iff integrate to zero against pairs}
	A function $q\in C^\infty_c(\B^n)$ is radial if and only if
	\eq{
	\label{eq: integral against harmonics}
		\int_{\B^n} q\cdot uv \d x=0
	}
	for every pair $u, v$ of homogeneous harmonic polynomials in $\R^n$ of distinct degree.
\end{lem}
\begin{proof}
	Suppose $q$ is radial. Recall that homogenous harmonic polynomials in $\R^n$ restrict to eigenfunctions of the spherical Laplacian on $\S^{n-1}$, and that the polynomial degree corresponds to the eigenvalue of the restricted function. Therefore, after rewriting the integral in spherical coordinates and applying the orthogonality of spherical harmonics of different eigenvalue, Equation \ref{eq: integral against harmonics} follows.

	For the implication in the other direction, we induct on the dimension. When $n=2$, view $\B^2=\mathbb{D}\subset \C$ as the unit disc in the complex plane. Our integration assumption is equivalent to the fact that
	\eq{
		\int_\mathbb{D} q z^k\overline{z}^\ell\d A=0
	}
	for every pair of \textit{distinct} integers $k,\ell\geq 0$. Decomposing $q$ radially into its Fourier series as $q=\sum_{j\in \Z}a_j(r)e^{ij\theta}$, we deduce that
	\aeq{
		0=\sum_{j\in \Z}\int_{\theta=0}^{2\pi} \int_{r=0}^1a_j(r)r^{k+\ell +1}e^{i(j+k-\ell)\theta}\d r\d \theta=\int_0^1a_{\ell-k}(r)r^{k+\ell+1}\d r.
	}
	It follows that all sufficiently high moments of $a_j$, $j\neq 0$, vanish, which in turn implies that $a_j\equiv 0$ when $j\neq 0$ so $q$ is radial.
	
	We now proceed with our induction and suppose that the result holds for some $n\geq 2$. Writing $\R^{n+1}=\R^n\times \R$ with corresponding coordinates $(x,y)$, let $P=\{y=0\}$ denote the horizontal hyperplane. Extending $q$ by zero outside of the ball, define
	\eq{
		q_P(x)=\int_{-\infty}^\infty q(x,y)\d y. 
	}
	Now, observe that for any pair $u,v\in \C^\infty (\R^n)$ of homogeneous harmonic polynomials of different degree,
	\eq{
		\int_{\R^n} q_P \cdot uv\d x= \int_{\R^{n+1}} q \cdot \tilde u\tilde v\d x \d y,
	}
	where $\tilde u(x,y)\coloneqq u(x)$ and similarly for $\tilde v$. Note that $\tilde u$ and $\tilde v$ themselves form a pair of homogeneous harmonic polynomials of different degree, now in $\R^{n+1}$. Conclude after applying our inductive hypothesis that $q_P$ is radial. This reasoning in fact holds for any hyperplane $P$, with $q_P$ the function obtained by integrating $q$ vertically over $P$.
	
	To conclude the argument, we apply the Fourier projection-slice theorem. Since $q_P$ is radial, so is $\widehat{q_P}$, which is exactly the restriction of $\widehat{q}$ to the plane $P$. It follows that $\widehat{q}$, and hence $q$, is radial.
\end{proof}

We now apply this result to conclude the promised theorem.

\begin{proof}[Proof of Theorem \ref{thm: conformal inf q2}]
	Fix $f,g\in C^\infty(\S^{n-1})$ and let $u, v_t\in C^\infty (\B^n)$ be the solutions to
	\eq{
	\begin{cases}
		\Delta u=0\\
		u=f \tn{ on }\S^{n-1}
	\end{cases}
	\quad\tn{and}\qquad
	\begin{cases}
		(\Delta +q_t)v_t=0\\
		v_t=g \tn{ on }\S^{n-1}
	\end{cases},
	}
	respectively. By Green's formula, we find that
	\eq{
		\ip{(\Lambda_{q_t}-\Lambda_0)f,g}_{\S^{n-1}}=\int_{\B^n} q_t u v_t,
	}
	which we differentiate at $t=0$ to obtain
	\eq{
		\ip{\Lambda' f,g}_{\S^{n-1}}=\int_{\B^n} q' u v_0.
	}
	Now, $\Lambda'$ commutes with $\Delta_{\S^{n-1}}$ if and only if $\ip{\Lambda' f,g}_{\S^{n-1}}=0$ whenever $f$ and $g$ are spherical harmonics with distinct eigenvalues. It follows that $q'$ integrates to zero against every pair of homogenous harmonic polynomials of different degree. The result thus immediately follows from Lemma \ref{lem: radial iff integrate to zero against pairs}.
\end{proof}

We conclude this section by stating as an open problem our expectation that Theorem \ref{thm: conformal inf q2} ought to hold for any potential (and thus conformal factor), not just those close to zero.

\begin{opr}
\label{opr: commute iff radial}
	Show that the Dirichlet-to-Neumann map associated to $(\B^n, g_\phi)$ commutes with the Laplacian on $\S^{n-1}$ if and only if $\phi$ is radial.
\end{opr}

\vspace{5em}

\appendix

\section{Appendix: Gohberg's Lemma}
\label{sec: Gohberg lemma}

Gohberg's Lemma, which serves as the starting step of the analysis of Section \ref{sec: euclidian}, provides a bound on the symbol of a pseudodifferential operator in terms of its norm. Throughout this section, we take $(M,g)$ to be a closed Riemannian manifold and denote by $\Psi^k_\tn{cl}(M)$ the space of classical pseudodifferential operators of order $k$ on $M$ (for the technical definition of this space and, more generally, an introduction to microlocal analysis, consult Chapter 18 of \cite{H85}).

Let $A\in \Psi^0_\tn{cl}(M)$ and denote its principal symbol by $\sigma_0(A)$. Observe that while $\sigma_0(A+B)=\sigma_0(A)$ for any $B\in \Psi^{-1}_\tn{cl}$, the norm $\norm{A+B}_2$ may be made arbitrarily large since smoothing operators have arbitrary norm on $L^2(M)$. However, one may hope for a lower bound on $\norm{A+B}_2$ in terms of the principal symbol, since $A$ cannot be ``corrected'' by a lower order operator to have arbitrarily small norm. Noting that all negative order pseudodifferential operators are compact by the Sobolev embedding theorem, this intuition leads precisely to Gohberg's lemma:

\begin{lem}[Gohberg's lemma]
\label{lem: Gohberg}
	For $A\in \Psi^0_\tn{cl}(M)$ we have
	\eq{
		\norm{\sigma_0(A)}_\infty\leq \norm{A}_\tn{ess},
	}
	where $\norm{A}_\tn{ess}$ is the \tn{essential norm} defined by
	\eq{
		\norm{A}_\tn{ess}\coloneqq \inf\{\,\norm{A-K}_2\tn{among all compact operators $K:L^2(M)\to L^2(M)$}\}.
	 }
\end{lem}

Results of this form originally appeared in \cite{G60}, but the modern literature does not seem to have an accessible account of the result. We therefore provide here a detailed exposition. We begin by investigating the model case in $\R^n$. To this end, let $a(x,\xi)\in C^\infty(\R^n\times (\R^n \setminus\{0\}))$ be compactly supported in $x$ and homogeneous of degree zero in $\xi$. Recall that its quantization $A=\tn{Op}(a)$ is initially defined for Schwartz functions $u\in \mc{S}(\R^n)$ by
\eq{
	Au(x)=\frac{1}{(2\pi )^n}\iint e^{i(x-y)\cdot \xi}a(x,\xi)u(y)\d y\d \xi
}
and extends by continuity to $L^2(\R^n)$.

\begin{prop}
\label{prop: mult op limit}
	Fix $\xi_0\in \R^n\setminus \{0\}$ and $f\in C^\infty_c(\R^n)$, and define $u_\lambda(x)=e^{i\lambda x\cdot \xi_0}f(x)$. Then
	\eq{
		\lim_{\lambda\to\infty}\norm{Au_\lambda-a(\cdot, \xi_0)u_\lambda}_2\to 0.
	}
\end{prop}

\begin{proof}
	The result is trivial if $f$ or $a$ is identically zero, so assume this is not the case. Compute first that
	\aeq{
		Au_\lambda(x)&=\frac{1}{(2\pi )^n}\int e^{ix\cdot \xi}a(x,\xi)\int e^{-iy\cdot (\xi-\lambda\xi_0)}f(y)\d y\d \xi\\
		&=\frac{1}{(2\pi )^n}\int e^{ix\cdot \xi}a(x,\xi)\widehat{f}(\xi-\lambda \xi_0)\d \xi\\
		&=e^{i\lambda x\cdot \xi_0}\frac{1}{(2\pi)^n}\int e^{ix\cdot \xi}a(x,\xi/\lambda+\xi_0)\widehat{f}(\xi)\d \xi,
	}
	where the last line is obtained by changing variables and applying the homogeneity property of $a$. Writing
	\eq{
		a(x,\xi_0)u_\lambda(x)=e^{i\lambda x\cdot \xi_0}\frac{1}{(2\pi)^n}\int e^{ix\cdot \xi}a(x,\xi_0)\widehat{f}(\xi)\d \xi
	}
	by taking the inverse Fourier transform of the Fourier transform of $f$, it follows that
	\eq{
	\label{eq:first bound}
	\abs{Au_\lambda(x)-a(x,\xi_0)u_\lambda(x)} \leq \frac{1}{(2\pi)^n}\int \abs{a(x,\xi/\lambda+\xi_0)-a(x,\xi_0)}\cdot \abs{\widehat{f}(\xi)}\d \xi.
	}
	We now show that the right hand side may be made uniformly arbitrarily small for $\lambda$ sufficiently large. To this end, let $\epsilon>0$ and note that since $f$ is nonzero, compactly supported and smooth, there exists $R>0$ such that
	\eq{
		\int_{\abs{\xi}\leq R} \abs{\widehat f(\xi)}\d \xi\neq 0\quad \tn{and}\quad\int_{\abs{\xi}>R} \abs{\widehat f(\xi)}\d \xi <\frac{\epsilon}{4\norm{a}_{\infty}}.
	}
	(Note that $0< \norm{a}_\infty<\infty$ since $a$ is nonzero, compactly supported in $x$ and homogeneous of degree zero in $\xi$.)
	Next, apply the uniform continuity of $a$ to pick $L$ sufficiently large so that for $\lambda>L$, we have for every $x$ and $\abs{\xi}\leq R$ that 
	\eq{
		|a\b{x,\xi/\lambda+\xi_0}-a(x,\xi_0)|<\epsilon\b{2\int_{\abs{\xi}\leq R} \abs{\widehat f(\xi)}\d \xi}^{-1}.
	}
	Conclude that for $\lambda>L$, the bound in Equation (\ref{eq:first bound}) becomes
	\aeq{
		(2\pi)^n\abs{A_\lambda u(x)-a(x,\xi_0)u_\lambda(x)}&<\int_{\abs{\xi}\leq R}\epsilon\b{2\int_{\abs{\xi}\leq R} \abs{\widehat f(\xi)}\d \xi}^{-1}\abs{\widehat f(\xi)}\d \xi\\
		&\qquad+\int_{\abs{\xi}>R}2\norm{a}_\infty\cdot \abs{\widehat f(\xi)}\d \xi\\
		&< \epsilon/2+\epsilon/2=\epsilon
	}
	as desired. The result now follows after noting that the support of $Au_\lambda-a(\cdot, \xi_0)u_\lambda$ is contained in a compact set independent of $\lambda$.
\end{proof}

We apply this proposition to obtain a lower bound on the norm of $A$ as an operator on $L^2(\R^n)$.

\begin{prop}
\label{prop: coordinate norm bound}
	With $A$ as above,
	\eq{
		\norm{a}_\infty\leq \norm{A}_2.
	}
\end{prop}

\begin{proof}
	Without loss of generality, let $(0,\xi_0)$ be a point where $\abs{a}$ is maximized. Take $\chi\in C^\infty_c(\R^n)$ a bump function with $\chi(0)=1$ and $\norm{\chi}_2=1$, and set 
	\eq{
		v_k(x)=k^{n/2}\chi(kx)e^{i\lambda_kx\cdot \xi_0},
	}
	where $\lambda_k$ is chosen so that $\norm{Av_k-a(\cdot, \xi_0)v_k}_2<1/k$ by Proposition \ref{prop: mult op limit}. Note that $\norm{v_k}_2=1$ by construction and that
	\aeq{
		\norm{Av_k-a(0,\xi_0)v_k}_2&\leq \norm{Av_k-a(\cdot, \xi_0)v_k}_2\\
		&\qquad \qquad +\b{\sup_{x\in \tn{supp}(v_k)}\abs{a(x,\xi_0)-a(0,\xi_0)}}\norm{v_k}_2\\
		&\to 0
	}
	since the support of $v_k$ shrinks to $0$. It follows that
	\eq{
		\norm{A}_2=\norm{A}_2\norm{v_k}_2\geq \norm{Av_k}_2\to \norm{a(0,\xi_0) v_k}_2=\norm{a}_\infty \norm{v_k}_2=\norm{a}_\infty,
	}
	as desired.
\end{proof}

This establishes the desired bound for operators in $\Psi^0_\tn{cl}(\R^n)$ whose symbol has compact support in the spatial variable. To obtain Gohberg's lemma, we  generalize this to closed manifolds by a standard microlocal procedure.

\begin{proof}[Proof of Lemma \ref{lem: Gohberg}]
	Let $(x_0,\xi_0)\in M\times (T^*M\setminus\{0\})$ be a point where $\abs{\sigma_0(A)}$ is maximized, and take $U_1$ to be a coordinate chart centered at this point. Complete $\{U_1\}$ to a finite cover $M=\cup_jU_j$ by coordinate charts $F_j\colon U_j\to \R^n$, and let $\{\phi_j\}$ be a partition of unity subordinate to $\{U_j\}$ such that $\sum_j \phi_j^2=1$. Following
	Chapter 18 of \cite{H85}, define
	\eq{
		B_ju=\phi_j F_j^*A_j(F_j^{-1})^*(\phi_ju),
	}
	where $A_j$ is the quantization of the pullback of $\sigma_0(A)$ to $T^*\R^n$ by $F_j^{-1}$. Set then $B=\sum B_j$, which will have principal symbol $\sigma_0(A)$. Note that $A-B$ is compact, so $\norm{A}_\tn{ess}=\norm{B}_\tn{ess}$ hence it sufficies to show the result for $B$.
	
	Observe that Proposition \ref{prop: mult op limit} holds for $B$ since although the operators $A_j$ do not have symbols with compact support, the symbols are still bounded in absolute value, and the $L^\infty$ bound can be converted to and $L^2$ bound as $M$ is assumed compact. Then, proceeding exactly as in the proof of Proposition \ref{prop: coordinate norm bound} within the $U_1$ coordinate chart, we obtain a sequence of functions $v_k$ such that
	\eq{
		\norm{Bv_k-a(x_0,\xi_0)v_k}_2\to 0.
	}
	The result follows after noting that $v_k\rightharpoonup 0$ weakly in $L^2$, hence $Kv_k\to 0$ for any compact $K$ thus, as earlier,
	\eq{
		\norm{A-K}_2\geq \norm{(A-K)v_k}_2\to \norm{a}_\infty,
	}
	thereby concluding the proof.
\end{proof}

To complete our discussion of Gohberg's lemma, we state the following easy consequence, which establishes a similar bound for classical pseudodifferential operators of any order.

\begin{cor}
\label{cor: Gohberg for any order}
	For $A\in \Psi^k_\tn{cl}(M)$ we have
	\eq{
		\norm{\,\abs{\xi}^{-k}\sigma_k(A)(x,\xi)}_\infty\leq \norm{A}_{H^k(M)\to L^2(M)}.
	}
\end{cor}
\begin{proof}
	The operator $A(\Delta+1)^{-k/2}$ is a classical pseudodifferential operator of order zero, and Lemma \ref{lem: Gohberg} thus states that
	\eq{
		\norm{\sigma_0\b{A(\Delta+1)^{-k/2}}}_\infty \leq \norm{A(\Delta+1)^{-k/2}}_\tn{ess}.
	}
	The result follows after noting that
	\eq{
		\sigma_0\b{A(\Delta+1)^{-k/2}}=\sigma_k(A)\sigma_{-k}\b{(\Delta+1)^{-k/2}}=\abs{\xi}^{-k}\sigma_k(A)
	}
	since principal symbols are multiplicative, and that
	\eq{
		\norm{A(\Delta+1)^{-k/2}}_\tn{ess}\leq \norm{A(\Delta+1)^{-k/2}}_2\leq \norm{A}_{H^k(M)\to L^2(M)}
	}
	since $(\Delta+1)^{-k/2}$ is definitionally as isometry from $L^2(M)$ to $H^k(M)$.
\end{proof}

\newpage

\bibliography{bib.bib}{}
\bibliographystyle{plain}

\end{document}